\documentclass{amsart}

\usepackage{amssymb,color}
\usepackage{amsfonts}
\usepackage{amsmath}
\usepackage{euscript}
\usepackage{enumerate}
\usepackage{graphics}

\newtheorem{theorem}{Theorem}[section]

\newtheorem{note}[theorem]{Note}

\newtheorem{cor}[theorem]{Corollary}
\newtheorem{exa}[theorem]{Example}

\newtheorem*{Theorem1'}{Theorem 1'}

\theoremstyle{definition}

\theoremstyle{remark}

\numberwithin{equation}{section}

\newcommand \g{{\mathfrak g}}

\newcommand \N{{\mathbb N}}

\newcommand \chr{{\mathrm {char}}}

\newcommand \Gal{{\mathrm {Gal}}}
\newcommand \End{{\mathrm {End}}}

\renewcommand \sl{{\mathfrak {sl}}}

\begin{document}

\title[Uniserial modules of associative and
Lie algebras]{On the theorem of the primitive element with
applications to the representation theory of associative and Lie
algebras}

\author{Leandro Cagliero}
\address{CIEM-CONICET, FAMAF-Universidad Nacional de C\'ordoba, C\'ordoba, Argentina.}
\email{cagliero@famaf.unc.edu.ar}
\thanks{The first author was supported in part by CONICET and SECYT-UNC grants.}


\author{Fernando Szechtman}
\address{Department of Mathematics and Statistics, University of Regina, Canada}
\email{fernando.szechtman@gmail.com}
\thanks{The second author was supported in part by an NSERC discovery grant}

\subjclass[2000]{Primary 17B10, 13C05; Secondary 12F10, 12E20}



\keywords{Uniserial module; Lie algebra; associative algebra;
primitive element}

\begin{abstract} We describe of all finite
dimensional uniserial representations of a commutative associative
(resp. abelian Lie) algebra over a perfect (resp. sufficiently
large perfect) field. In the Lie case the size of the field
depends on the answer to following question, considered and solved
in this paper. Let $K/F$ be a finite separable field extension and
let $x,y\in K$. When is $F[x,y]=F[\alpha x+\beta y]$ for some
non-zero elements $\alpha,\beta\in F$?
\end{abstract}

\maketitle

\section{Introduction}

That the classification of all finite dimensional indecomposable
modules over almost any Lie algebra is a hopeless enterprise
follows from \cite{Ma} (see \cite{GP} for the 2-dimensional
abelian case).

It is more realistic to concentrate on certain indecomposable
modules over some types of Lie algebras. In this regard, recall
that a module is uniserial if it is non-zero and its submodules
are totally ordered by inclusion. In this paper we consider the
problem of describing all finite dimensional uniserial
representations of an abelian Lie algebra over an arbitrary field.

This is the starting point of a project aiming to systematically
investigate the finite dimensional uniserial representations of
distinguished classes of Lie algebras. We have recently classified
\cite{CS1} all such representations for the family of complex
perfect Lie algebras $\sl(2)\ltimes V(m)$, where $V(m)$ is the
irreducible $\sl(2)$-module of highest weight $m\geq 1$. This
classification turned out to rely on certain zeros of the
Racah-Wigner $6j$-symbol. The analogous problem for a family of
solvable Lie algebras over an arbitrary field is considered in
\cite{CS2}.

In this context, our main result is the following.

\begin{theorem}\label{thm.int} Let $\g$ be an abelian Lie algebra over a perfect field
$F$ and let $V$ be a finite dimensional uniserial $\g$-module. Let
$\ell$ be the composition length of $V$ and let $W$ be the socle
of $V$.  Let $N$ be the number of distinct prime factors of
$\dim_F(W)$ and suppose that $|F|>N-1$. Then there exists $x\in\g$
such that $V$ is a uniserial $F[x]$-module. In particular, $x$
acts on $V$, relative to some basis, via the companion matrix
$C_f$ of a power $f=p^\ell$ of an irreducible polynomial $p\in
F[X]$, and every other element of $\g$ acts on $V$ via a
polynomial on $C_f$.
\end{theorem}

The condition that $F$ be a perfect field such that $|F|>N-1$ is
essential. A full account of what happens when $F$ is imperfect is
given in Note \ref{menti}. The condition $|F|>N-1$ is related to
the following:

\medskip

\noindent{\sc Question A.} Let $K/F$ be a finite separable field
extension and let $x,y\in K$. When is $F[x,y]=F[\alpha x+\beta y]$
for some non-zero elements $\alpha,\beta\in F$?

\medskip

The same question with \emph{some} replaced by \emph{all} has a
long history, which is explored in detail in \S\ref{bas}. For the
stated version, we have the following results.

\begin{theorem}\label{qintro} Let $K/F$ be a finite separable field extension,
and let $x,y\in K$ have respective degrees $a=md$ and $b=nd$ over
$F$, where $d=\gcd(a,b)$. Let $N$ be the number of distinct prime
factors of $d$ not dividing $mn$. If $|F|>N+1$ then there are
non-zero $\alpha,\beta\in F$ such that $F[x,y]=F[\alpha x+\beta
y]$.
\end{theorem}

\begin{theorem}\label{intro:be} Let $K/F$ be a finite separable field extension.
Let $N$ be the number of distinct prime factors of $[K:F]$.
Suppose that $|F|>N-1$. Then, given any $x_1,\dots,x_n\in K$ such
that $K=F[x_1,\dots,x_n]$, there is an $F$-linear combination
$z=\alpha_1x_1+\cdots+\alpha_n x_n$ such that all $\alpha_i\neq 0$
and $K=F[z]$.
\end{theorem}

That the conditions involving $|F|$ in Theorems \ref{thm.int},
\ref{qintro} and \ref{intro:be} are exact is shown in Theorem
\ref{unomas}.

We remark that Theorem \ref{thm.int} is a corollary of a closely
related result concerning associative algebras, which reads as
follows.

\begin{theorem}
\label{abelintro} Let $F$ be a perfect field and let $A$ be a
finite dimensional commutative and associative algebra over $F$.
Then the following conditions are equivalent:

(a) $A=F[u]$ for some $u\in A$ whose minimal polynomial over $F$
is an $\ell$-power of an irreducible polynomial in $F[X]$.

(b) The regular module of $A$ is uniserial of length $\ell$.

(c) $A$ has a finite dimensional faithful uniserial representation
of length $\ell$.

Moreover, suppose any of the conditions (a)-(c) is satisfied. Then
$A$ has a unique irreducible module, up to isomorphism, namely the
residue field, say $R(A)$, of $A$. Let $N$ be the number of
distinct prime factors of $[R(A):F]$ and suppose that $|F|>N-1$.
Then, given any elements $x_1,\dots,x_n\in A$ such that
$A=F[x_1,\dots,x_n]$, there is an $F$-linear combination $z$ of
$x_1,\dots,x_n$ such that $A=F[z]$.
\end{theorem}

Our proof of Theorem \ref{abelintro}, given in \S\ref{sec:uni}, is
fairly subtle, a key ingredient being the existence and uniqueness
of the Jordan-Chevalley decomposition of an endomorphism acting on
a finite dimensional vector space over a perfect field. Another
consequence of Theorem \ref{abelintro} is the following
characterization of the finite dimensional uniserial modules of a
commutative and associative algebra over a perfect field.

\begin{theorem}\label{t1intro} Let $F$ be a perfect field and let $A$ be a commutative and associative $F$-algebra.
Let $V$ be a finite dimensional uniserial $A$-module of length
$\ell$. Then there exists $x\in A$ such that $V$ is a uniserial
$F[x]$-module. In particular, $x$ acts on $V$ via the companion
matrix $C_f$ of a power $f=p^\ell$ of an irreducible polynomial
$p\in F[X]$, and every element of $A$ acts on $V$ via a polynomial
on $C_f$.
\end{theorem}

To place Theorems \ref{thm.int}, \ref{abelintro} and \ref{t1intro}
in context, we observe that the class of uniserial modules is very
important for associative algebras, whereas for Lie algebras this
class has been barely considered. We are confident that it is as
relevant as in the associative case and thus worthy to be studied
in detail.

There is an extensive literature on uniserial modules and, more
generally, on uniserial and serial rings. Recall that a ring $R$
with identity is called \emph{uniserial} (resp.\ \emph{serial}) if
both $R_R$ and ${}_RR$ are uniserial (resp.\ direct sum of
uniserial) modules. Here all modules are assumed to be unitary.

Serial rings and algebras occur in several contexts. This class
includes discrete valuation rings, Nakayama algebras, triangular
matrix rings over a skew field and Artinian principal ideal rings
(see \cite{Pu}, \cite{EG}). In particular, every proper factor
ring of a Dedekind domain is serial. Also, serial algebras occur
as the group algebras in characteristic $p$ of certain finite
groups, including all $p$-solvable groups with cyclic Sylow
$p$-subgroups (see \cite{Sr}).

Some important results include the following.
\begin{enumerate}[(1)]
 \item  It is due to T.\ Nakayama \cite[Theorem 17]{Na} that
a finitely generated module over a serial ring is a serial module
(i.e., a direct sum of uniserial modules).
\medskip
\item Nakayama algebras (those whose right and left projective
modules are uniserial) are of finite representation type, i.e.,
they have only a finite number of indecomposable modules up to
isomorphism (see \cite[Ch.\ VI, Theorem 2.1]{ARS}).
\medskip
\item D. Eisenbud and P. Griffith \cite{EG} proved that:
\begin{enumerate}[(i)]
\item For a finite  dimensional $F$-algebra,  the  serial property
is stable under  a change of the base field, provided that
$A_0=A/\text{Rad}(A)$ be separable, i.e., $A_0\otimes_F K$ is
semisimple for every field extension $K/F$.

\item The composition series of any uniserial module over a serial
ring is periodic in a strong sense (see \cite{EG}, Theorem 2.3,
for a precise statement).

\item Any two simple modules of an indecomposable uniserial ring
have the same endomorphism ring.
\end{enumerate}

\medskip

\item A complete description of finite dimensional serial algebras
over a perfect field is contained in \cite{AF} and \cite{DK}.
\medskip


\item T. Shores and W. Lewis \cite{SL} show that if $M$ is a
faithful uniserial module over an integral domain $R$, then
$\text{End}_R(M)$ is a valuation ring.
\end{enumerate}


\section{Finding primitive elements}\label{bas}

Let $K/F$ be an algebraic separable field extension and let
$x,y\in K$.

\noindent{\sc Question 1.} When is $F[x,y]=F[\alpha x+\beta y]$
for {\em all} non-zero elements $\alpha,\beta\in F$?

\noindent{\sc Question 2.} When is $F[x,y]=F[\alpha x+\beta y]$
for {\em some} non-zero elements $\alpha,\beta\in F$?

Question 1 has been studied by many authors and a brief summary is
given below. However, for our purposes, Question 1 only plays a
subsidiary role to Question 2, which is an indispensable tool in
understanding the finite dimensional uniserial representations of
abelian Lie algebras.

By the standard proof of the theorem of the primitive element (see
\cite{Ar}, Theorem 26) we may restrict our consideration of
Question 2 to the case when $F$ is a finite field. In this case,
as is well-known, every finite extension of $F$ has a primitive
element, but it may surprise the reader to learn that one can
always find algebraic elements $x,y$ such that $F[x,y]\neq
F[\alpha x+\beta y]$ for {\em all} $\alpha,\beta\in F$. Examples
are not so easy to construct. It is also unclear at first how to
find the exact conditions under which the existence of
$\alpha,\beta\in F^\times $ such that $F[x,y]=F[\alpha x+\beta y]$
is guaranteed. When these conditions fail, our general description
of the finite dimensional uniserial representations of abelian Lie
algebras over $F$ ceases to be true.

Our answer to Question 2 is given in Theorems \ref{q} and
\ref{be}, while Theorem \ref{unomas} shows that the conditions are
exact. Our main applications of the results of this section to
representation theory are given in \S\ref{sec:uni}.

Question 1 seems to have been first considered by Nagell
\cite{N1}, \cite{N2}, followed by Kaplansky \cite{K} and Isaacs
\cite{I}. Let $x$ and $y$ have degrees $m$ and $n$ over $F$,
respectively. Isaacs showed that if $\gcd(m,n)=1$ then
$F[x,y]=F[\alpha x+\beta y]$ for all $\alpha,\beta\in F^\times$,
provided certain technical conditions hold when $F$ has prime
characteristic $p$. The same conclusion was later obtained by
Browkin, Divi$\rm{\breve{s}}$ and Schinzel \cite{BDS}, provided
$[F[x,y]:F]=mn$ and different technical conditions hold in the
prime characteristic case.

Recently, Weintraub \cite{W} showed that if $F[x]/F$ and $F[y]/F$
are Galois extensions, $F[x]\cap F[y]=F$ and $p\nmid \gcd(m,n)$
when $F$ has prime characteristic $p$, then $F[x,y]=F[\alpha
x+\beta y]$ for all $\alpha,\beta\in F^\times$. We also obtained
and proved this result (see Theorem \ref{gen}), unaware at the
moment of Weintraub's paper. The condition $p\nmid \gcd(m,n)$ is
essential, as seen in Example \ref{pedo}. In particular, if
$F[x,y]/F$ is an abelian Galois extension and $\gcd(m,n)=1$ then
$F[x,y]=F[\alpha x+ \beta y]$ for all $\alpha ,\beta \in F^\times$
(see Corollary \ref{cw}). In this regard, see \cite{P}, \S 3.

When $F[x]\cap F[y]\neq F$ an answer to Question 1 is more
difficult. Sufficient conditions are given in Theorem \ref{t},
which provides enough information for us to answer Question 2 in
Theorems \ref{q} and \ref{be} when $F$ is a finite field.

\begin{theorem}\label{gen} Let $K/F$ be an algebraic extension.
Let $x$ and $y$ be elements of $K$ of respective degrees $m$ and
$n$ over $F$, and satisfying:

(C1) $F[x]/F$ and $F[y]/F$ are Galois extensions.

(C2) $F[x]\cap F[y]=F$.

(C3) If $F$ has prime characteristic $p$ then $p\nmid \gcd(m,n)$.

Then $F[x,y]=F[\alpha x+\beta y]$ for all $\alpha ,\beta\in F$
different from 0.
\end{theorem}

\begin{proof} It suffices to prove the result for $\alpha =\beta=1$. It is
well-known (\cite{DF}, \S 14.4, Corollary 22) that $F[x,y]/F$ is a
Galois extension of degree $mn$.

For $z\in K$ let $S_z$ be the stabilizer of $z$ in $\Gal(K/F)$. We
need to show that $S_{x+y}$ is trivial. Since $x,y$ generate
$F[x,y]$ over $F$, this is equivalent to $S_{x+y}\subseteq S_x\cap
S_y$.

Let $\sigma\in S_{x+y}$. Then
$$
\sigma(x+y)=x+y,
$$
so
\begin{equation}
\label{crux} \sigma(x)-x=-(\sigma(y)-y).
\end{equation}
Since $F[x]/F$ and $F[y]/F$ are normal, this common element, say
$f$, belongs to $F[x]\cap F[y]=F$.

Suppose first $\chr(F)=0$. Then $\sigma(x)=x+f$ implies $f=0$,
since $\sigma$ has finite order. It follows that $\sigma\in
S_x\cap S_y$.

Suppose next $F$ has prime characteristic $p$. Without loss of
generality we may assume that $p\nmid m$. From
$$
\sigma(x)=x+f
$$
we infer
$$
\sigma^{p}(x)=x,
$$
so $\sigma^p\in S_x$. Now $G_x$ is normal in $G$ and $G/G_x$ has
order $m$, so $\sigma^m\in G_x$. Since $\gcd(p,m)=1$, we deduce
$\sigma\in S_x$, which implies $f=0$ and a fortiori $\sigma\in
S_y$, as required.
\end{proof}

\begin{cor}\label{cw} Let $K/F$ be a finite abelian Galois extension. If $x,y\in K$ have degrees $m,n$ over $F$
and $\gcd(m,n)=1$ then $F[x,y]=F[\alpha x+ \beta y]$ for all
$\alpha ,\beta \in F^\times$.
\end{cor}

The condition that $G=\mathrm{Gal}(K/F)$ be abelian in Corollary
\ref{cw} is placed to ensure that every subgroup of $G$ be normal.
The groups for which this condition holds are called Dedekind
groups. Every Dedekind group is the direct product of the
quaternion group $Q_8$ with an abelian group \cite{R}, so not much
is gained by relaxing our hypothesis so that $G$ be a Dedekind
group.

Condition (C3) of Theorem \ref{gen} is redundant if $F$ is a
finite field or, more generally, if the Galois group of $F[x,y]/F$
is cyclic. Indeed, in this case $$[F[x]\cap F[y]:F]=\gcd(m,n),$$
so $F[x]\cap F[y]=F$ if and only if $\gcd(m,n)=1$.

However, (C3) cannot be dropped entirely, as the following example
shows. This example also illustrates the fact that the element
(\ref{crux}) of $F$ need not be zero if (C3) fails.

\begin{exa}\label{pedo}{\rm Given a prime $p$, let $E=F_{p^2}$. Let $X,Z$ be algebraically
independent over $E$ and let $q=Z^{p^2}-Z-X\in E[X,Z]$. It is not
difficult to see that $q$ is irreducible in $E[Z][X]$, hence in
$E[X,Z]$ and therefore in $E[X][Z]$. Let $F=E(X)$. It follows from
Gauss' Lemma that $q\in F[Z]$ is irreducible.  Let $\alpha$ be a
root of $q$ in some extension of~$F$ and set $K=F[\alpha]$. Then
$\alpha+a\in K$ is also a root of $q$ for every $a\in E$, so $K/F$
is a Galois extension of degree $p^2$. Let $G=\mathrm{Gal}(K/F)$.
For each $a\in E$ let $\sigma_a\in G$ be given by $\alpha\mapsto
\alpha+a$. Then $a\mapsto\sigma_a$ is a group isomorphism from the
underlying additive group of $E$ onto $G$. In particular, $G\cong
C_p\times C_p$.

Let $$x=\alpha(\alpha+1)\times\cdots \times(\alpha+(p-1)).$$ Then
\begin{equation}
\label{ex} x=\alpha^p-\alpha.
\end{equation}
Since $\alpha$ has degree $p^2$ over $F$, it follows that $x$ is
not in $F$. On the other hand,
$$
x^p+x-X=\alpha^{p^2}-\alpha^p+\alpha^p-\alpha-X=\alpha+X-\alpha-X=0.
$$
Thus $x$ is a root of $Z^p+Z-X\in F[Z]$. Since $[F[x]:F]>1$ and
$[K:F]=p^2$, it follows that $[F[x]:F]=p$.

Since $G$ is abelian, $F[x]/F$ is a Galois extension of degree
$p$. Here $F[x]$ is the fixed field of $\sigma\in G$, given by
$\alpha\mapsto \alpha+1$. In fact,
$$
x=\alpha\times
\alpha^\sigma\times\cdots\times\alpha^{\sigma^{p-1}}.
$$

Next take $a$ in $E$ but not in $F_p$, and let
$$
y=\alpha(\alpha+a)\times\cdots \times(\alpha+(p-1)a).
$$
Then for $\beta=\alpha/a$, we have
\begin{equation}
\label{ey} y=a^p y/a^p=a^p \beta(\beta+1)\times\cdots
(\beta+(p-1))=a^p(\beta^p-\beta)=\alpha^p-a^{p-1}\alpha.
\end{equation}
In particular, $y$ is not in $F$. On the other hand,
$$
y^p+a^{1-p}y-X=\alpha+X-a^{1-p}\alpha^p+a^{1-p}\alpha^p-\alpha-X=0.
$$
Thus, $y$ has minimal polynomial $Z^p+a^{1-p}Z-X$ over $F$, the
extension $F[y]/F$ is Galois of degree $p$, and $F[y]$ is the
fixed field of $\tau\in G$, given by $\alpha\mapsto \alpha+a$,
with
$$
y=\alpha\times \alpha^\tau\times\cdots\times\alpha^{\tau^{p-1}}.
$$
Note that $\sigma$ and $\tau$ both have order $p$ but, by the
choice of $a$, $\langle\sigma\rangle\cap \langle\tau\rangle$ is
trivial. In particular, $F[x]$ and $F[y]$ are different subfields
of $K$. Degree considerations imply that $F[x]\cap F[y]=F$.

We wish to find all $b,c\in F$ different from 0 such that
$K=F[bx+yc]$. Clearly this is equivalent to finding all $b\in F$,
$b\neq 0$, such that $K=F[x+by]$, and we claim that this holds if
and only if $b\notin F_p\cdot a$. Thus, all $x+fay$, where $f\in
F_p$ is non-zero, have degree $p$ over $F$, while all $x+by$,
where $b\in F$ is not an $F_p$-scalar multiple of~$a$, have degree
$p^2$ over $F$.

To see the claim, note that first of all that $G$ consists of all
$\sigma^i\tau^j$, where $i,j\in F_p$. Now (\ref{ex}) gives
$$
x^{\sigma^i\tau^j}=(\alpha+ja)^p-(\alpha+ja)=x+j(a^p-a),
$$
while (\ref{ey}) yields
$$
(by)^{\sigma^i\tau^j}=b[(\alpha+i)^p-a^{p-1}(\alpha+i)]=b[y+i(1-a^{p-1})]=by+ib(1-a^{p-1}).
$$
Therefore
$$
(x+by)^{\sigma^i\tau^j}=x+by+ja(a^{p-1}-1)+ib(1-a^{p-1}).
$$
By the choice of $a$, we have $a^{p-1}\neq 1$. Therefore
$\sigma^i\tau^j$ fixes $x+by$ if and only if $ja=ib$. Clearly
$i,j$ are both zero or both non-zero. In the first case we deal
with the identity element of $G$, which fixes every element of
$K$. In the second case $\sigma^i\tau^j$ fixes $x+by$ if and only
if $b=j/i\times a$. All in all, $x+by$ has non-trivial stabilizer
in $G$ if and only if $b\in F_p\cdot a$, as claimed. }
\end{exa}

Further information about generalized Artin-Schreier polynomials
of the form $Z^{p^n}-Z-a$ and their associated Galois extensions,
including the computation of primitive elements for every
intermediate field, can be found in \cite{GS}.

\begin{theorem}\label{t}
Let $K/F$ be a field extension, let $x,y\in K$ of degrees $a=md$
and $b=md$, respectively, with $d=\gcd(a,b)$, and let $E=F[x]\cap
F[y]$. Given $\alpha,\beta\in F^{\times}$ let $z=\alpha x+\beta y$
and let $p_{z,E}$ be the minimal polynomial of $z$ over $E$.
Assume that:
\begin{enumerate}[(i)]
 \item $F[x]/F$ and $F[y]/F$ are finite Galois extensions.
 \item $K=F[x,y]$ has degree $mnd=\mathrm{lcm}(a,b)$ over $F$.
\end{enumerate}
Then the degree of $z$ over $F$ is $mne$, with $e=d/|S|$ and $S$
is the stabilizer of $p_{z,E}$ in $\Gal(E/F)$. Moreover, if $K/F$
is cyclic and $p|mn$ for every factor $p$ of $d$ then $K=F[z]$. In
particular, if $K/F$ is cyclic and $d|mn$ then $K=F[z]$.
\end{theorem}
\begin{proof} That $K/F$ are $E/F$ are Galois extensions follows from \cite{DF},
\S 14.4, Proposition 21. Here $[E:F]=d$, by \cite{DF}, \S 14.4,
Corollary~20.

Now $K/E$ is a Galois extension (\cite{DF}, \S 14.2, Theorem 14)
of degree $mn$, where $\gcd(m,n)=1$. Moreover, $E[x]=F[x]$ has
degree $m$ over $E$ and $E[y]=F[y]$ has degree $n$ over $E$. It
follows from Theorem \ref{gen} that $K=E[z]$. Thus the degree $z$
over $E$ is $mn$. If $T$ is a transversal for $S$ in $\Gal(E/F)$
then
$$
\underset{\sigma\in T}\Pi p_{z,E}^\sigma
$$
is the minimal polynomial of $z$ over $F$ and therefore the degree
of $z$ over $F$ is $mn|\Gal(E/F)|/|S|=mnd/|S|$.

Suppose next $\mathrm{Gal}(K/F)=\langle\sigma\rangle$ is cyclic.
We claim that if $p$ is a prime and $p|mn$ then
$\sigma^{mnd/p}(z)\neq z$. Indeed, if $p|m$ then
$\sigma^{mnd/p}(y)=\sigma^{bm/p}(y)=y$ and, since $\gcd(m,n)=1$,
we infer $p\nmid n$, so $\sigma^{mnd/p}(x)=\sigma^{an/p}(x)\neq
x$, whence $\sigma^{mnd/p}(z)\neq z$. The case $p|n$ is analogous.
This proves the claim. Consequently, if every prime factor $p$ of
$d$ divides $mn$ then necessarily, $S=1$, so $K=F[z]$.
\end{proof}

\begin{theorem}\label{q} Let $K/F$ be a finite separable field
extension, and let $x,y\in K$ have respective degrees $a=md$ and
$b=nd$ over $F$, where $d=\gcd(a,b)$. Let $A$ be the number of
distinct prime factors of $d$ not dividing $mn$. If $|F|>A+1$ then
there is $0\neq\alpha\in F$ such that $F[x,y]=F[x+\alpha y]$.
\end{theorem}

\begin{proof} By the standard proof of the theorem of
the primitive element (see \cite{Ar}, Theorem 26) we may assume
that $F$ is finite.

Let $K=F[x,y]$. Since $K/F$ is cyclic, we see that $[F[x]\cap
F[y]:F]=d$ and $[K:F]=\mathrm{lcm}(a,b)$.

By Theorem \ref{t}, if $\alpha\in F$, $\alpha\neq 0$, then
$x+\alpha y$ has degree $mne_\alpha$ over $F$ for some
$e_\alpha|d$. We wish to select $\alpha$ so that $e_\alpha=d$.

We have $\mathrm{Gal}(K/F)=\langle \sigma\rangle$, where $\sigma$
has order $mnd$. Suppose $p$ is a prime factor of $d$. If $p|mn$
then $\sigma^{mnd/p}(x+\alpha y)\neq x+\alpha y$, while if $p\nmid
mn$ then $\sigma^{mnd/p}(x)\neq x$ and $\sigma^{mnd/p}(y)\neq y$,
as seen in the proof of Theorem \ref{t}. In the latter case there
is at most one $\alpha\neq 0$ in $F$ such that $\sigma^{mnd/p}(x+
\alpha y)=x+ \alpha y$.

If $|F|-1>A$ then there is $\alpha\neq 0$ in $F$ such that
$x+\alpha y\in$ is not fixed by $\sigma^{mnd/p}$ for any prime
factor $p$ of $d$ such that $d\nmid mn$, so $e_\alpha=d$ in this
case.
\end{proof}

\begin{theorem}\label{be} Let $K/F$ be a finite separable field extension.
Let $B$ be the number of distinct prime factors of $[K:F]$.
Suppose that $|F|>B-1$. Then, given any $x_1,\dots,x_\ell\in K$
such that $K=F[x_1,\dots,x_\ell]$, there is an $F$-linear
combination $z=\alpha_1x_1+\cdots+\alpha_\ell x_\ell$ such that
all $\alpha_i\neq 0$ and $K=F[z]$.
\end{theorem}

\begin{proof} By the standard proof of the theorem of
the primitive element (see \cite{Ar}, Theorem 26) we may assume
that $F$ is finite. In this case, we argue by induction on $\ell$.
If $\ell=1$ there is nothing to do.

Suppose $\ell=2$ and let $p_1,\dots,p_B$ be the distinct prime
factors of $[K:F]$. Let $a=p_1^{a_1}\cdots p_B^{a_B}$ and
$b=p_1^{b_1}\cdots p_s^{b_s}$ be the respective degrees of $x_1$
and $x_2$ over~$F$, where $a_i,b_j\geq 0$. If all $a_i\leq b_i$ we
take $z=x_2$ and if all $b_j\leq a_j$ we take $z=x_1$. Suppose
next $a_i>b_i$ and $b_j>a_j$ for some $1\leq i\neq j\leq B$. Let
$d=\gcd(a,b)$, $m=a/d$, $n=b/d$, and let $A$ be the number of
distinct prime factors of $d$ not dividing $mn$. Since $p_i|m$ and
$p_j|n$, we have $A\leq B-2$. Therefore $|F|-1>B-2\geq A$, so
$K=F[z]$ for some $F$-linear combination
$z=\alpha_1x_1+\alpha_2x_2$, $\alpha_i\neq 0$, by Theorem \ref{q}.

Suppose $\ell>2$ and the result is true for $\ell-1$. Let $C$ be
the number of distinct prime factors of
$[F[x_1,\dots,x_{\ell-1}]:F]$. Then $C\leq B$, so $|F|>C-1$. By
inductive assumption there is an $F$-linear combination
$u=\beta_1x_1+\cdots+\beta_{\ell-1}x_{\ell-1}$, $\beta_i\neq 0$,
such that $F[x_1,\dots,x_{\ell-1}]=F[u]$. Therefore
$K=F[u,x_\ell]$. By the case $\ell=2$ we can find an $F$-linear
combination $z=\gamma_1u+\gamma_2x_\ell$, $\gamma_i\neq 0$, such
that $K=F[z]$. Since $z=\alpha_1x_1+\cdots+\alpha_\ell x_\ell$,
with all $0\neq \alpha_i\in F$, the proof is complete.
\end{proof}

\begin{theorem}\label{unomas} Given a prime power $q$ and $A\geq q-1$, there exist $a,b\geq 1$
and $x,y$ in some extension of $F=F_q$ such that: $d=\gcd(a,b)$
has exactly $A$ distinct prime factors which are not factors of
$a/d\times b/d$; $x$ has degree $a$ over $F$; $y$ has degree $b$
over $F$; no $F$-linear combination of $x,y$ generates $F[x,y]$
over $F$.
\end{theorem}

\begin{proof} All degrees considered below are
taken over $F=F_q$. Let $(p_i)_{i\in F^\times}$ be a family of
distinct primes in $\N$, and let $r,s$ be primes in $\N$ distinct
from each other and all $p_i$. Let $d$ be the product of all $p_i$
and let $t=drs$. Let $K$ be an extension of $F$ of degree $t$. For
$i\in F^\times$, let $u_i\in K$ be an element of degree $p_i$. We
further take $x_0,y_0\in K$ of degree $r$ and $s$, respectively.
Let $x$ be the sum of all $iu_i$ and $x_0$, and let $y$ be the sum
of all $u_i$ and $y_0$. It follows from Corollary \ref{cw} that
$x,y$ have respective degrees $a=t/s$ and $b=t/r$, whence
$K=F[x,y]$. Moreover, $\gcd(a,b)=d$. The number of distinct prime
factors of $d$ that are not factors of $a/d\times b/d=rs$ is
$q-1$. We claim that no $F$-linear combination of $x,y$ generates
$K$ over $F$. It suffices to verify this for $x,y$ and all $x-iy$,
$i\in F^\times$. Now $s$ does not divide the degree of $x$, $r$
does not divide the degree of of $y$, and $p_i$ does not divide
the degree of $x-iy$, as required.

If $A>q-1$, we can easily modify the above construction by
selecting $k=A-(q-1)$ further primes, say $q_1,\dots,q_k$,
multiplying the previous choices of $d$ and $t$ by $q_1\cdots
q_k$, selecting $v_1,\dots,v_k\in K$ of respective degrees
$q_1,\dots,q_k$, and adding the sum of the $v_i$ to the previous
choices of $x$ and $y$. The same conclusion follows.
\end{proof}

\section{Uniserial representations of abelian associative and Lie algebras}
\label{sec:uni}

\begin{theorem}
\label{abel} Let $F$ be a perfect field and let $A$ be a finite
dimensional commutative and associative algebra over $F$. Then the
following conditions are equivalent:

(a) $A=F[u]$ for some $u\in A$ whose minimal polynomial over $F$
is an $\ell$-power of an irreducible polynomial in $F[X]$.

(b) The regular module of $A$ is uniserial of length $\ell$.

(c) $A$ has a finite dimensional faithful uniserial representation
of length $\ell$.

Moreover, suppose any of the conditions (a)-(c) is satisfied. Then
$A$ has a unique irreducible module, up to isomorphism, namely the
residue field, say $R(A)$, of $A$. Let $N$ be the number of
distinct prime factors of $[R(A):F]$ and suppose that $|F|>N-1$.
Then, given any elements $x_1,\dots,x_n\in A$ such that
$A=F[x_1,\dots,x_n]$, there is an $F$-linear combination $z$ of
$x_1,\dots,x_n$ such that $A=F[z]$.
\end{theorem}

\begin{proof} It is obvious that (a) implies (b) and that (b)
implies (c).

Suppose (c) holds. Then there is a finite dimensional $F$-vector
space $V$ such that $A$ is a subalgebra of $\End(V)$ and $V$ is a
uniserial $A$-module. We wish to show that there is $u\in A$ such
that $A=F[u]$, where the minimal polynomial of $u$ is an
$\ell$-power of an irreducible polynomial in $F[X]$.

We argue by induction on the composition length $\ell$ of $V$ as
an $A$-module. Suppose first that $\ell=1$, that is, $V$ is
irreducible. It follows from Schur's Lemma that
$D=\mathrm{End}_A(V)$ is a finite dimensional division algebra
over $F$. Since $A$ is commutative, every $v\mapsto av$, $a\in A$,
is in $D$. Thus the image of $A$ in $D$ is a finite field
extension of $F$, and hence so is $A$ since $V$ is faithful. Since
$F$ is perfect, there is $u\in A$ such that $A=F[u]$ (see
\cite{Ar}, Theorem 27). Clearly, the minimal polynomial of $u$ is
irreducible in $F[X]$.

Suppose next that $\ell>1$ and the result is true for faithful
uniserial modules of length $<\ell$.

Let $a\in A$. Since $A$ is commutative and $V$ is uniserial over
$A$, the minimal polynomial of $a$ must be a power of monic
irreducible polynomial $p_a\in F[X]$. Moreover, the minimal
polynomial of $a$ acting on any finite dimensional irreducible
$A$-module must be irreducible. We deduce that the minimal
polynomial of $a$ acting on any composition factor of $V$ is
$p_a$.

Let $W$ be the socle of $V$. Since $V$ is uniserial, $W$ is
irreducible. Let $B$ the subalgebra of $\End(W)$ consisting of all
$a|_W$, $a\in A$. By above, $B=F[b|_W]$ for some $b\in A$. Since
$W$ is irreducible as a module for $F[b]$, we see that
$\dim(W)=\deg(p_b)$.

Let $C$ be the subalgebra of $\End(V/W)$ of all $\widetilde{a}$,
where $a\in A$ and $\widetilde{a}(v+W)=a(v)+W$ for $v\in V$. Note
that $V/W$ is a faithful uniserial $C$-module of length $\ell-1$.
By inductive hypothesis, there is $x\in A$ such that
$C=F[\widetilde{x}]$, where the minimal polynomial of $x$ in $V/W$
is $p_x^{\ell-1}$.

We claim that $\deg(p_x)=\deg(p_b)$. Indeed, from $x|_W\in
F[b|_W]$ we deduce $\deg(p_x)\leq\deg(p_b)$. Let $U/W$ be the
socle of $V/W$, which is an irreducible $A$-module. Let $\hat{x}$
and $\hat{b}$ be the elements of $\End(U/W)$ corresponding to $x$
and $b$, respectively. Since $\hat{b}\in F[\hat{x}]$, we infer
$\deg(p_b)\leq\deg(p_x)$, as claimed.

From $\deg(p_x)=\deg(p_b)=\dim(W)$ we deduce that $W$ is an
irreducible $F[x]$-module. From now on we let $p=p_x$. Thus $x$
has a single elementary divisor $p^{\ell-1}$ on $V/W$ and a single
elementary divisor $p$ on $W$. Therefore, one of the following
cases must occur:

\noindent{\sc Case 1.} $x$ has a single elementary divisor
$p^\ell$. In this case we take $u=x$. Since every $a\in A$
commutes with the cyclic operator $u$, we infer $A=F[u]$.

\noindent{\sc Case 2.} $x$ has elementary divisors $p^{\ell-1}$
and $p$. Let $R=F[X]$ and view $V$ as an $R$-module via $x$. Then
$V=Rv_1\oplus Rv_2$, where the annihilating ideals of $v_1$
and~$v_2$ are respectively generated by $p$ and $p^{\ell-1}$.
Suppose, if possible, that $\ell>2$. Let $a\in A$. Then
$av_2=g(x)v_2+h(x)v_1$ for some $g,h\in R$. Therefore
$$
a p^{\ell-2}(x)v_2=
p^{\ell-2}(x)av_2=g(x)p^{\ell-2}(x)v_2+h(x)p^{\ell-2}(x)v_1=g(x)p^{\ell-2}(x)v_2.
$$
Thus $Rp^{\ell-2}(x)v_2$ is an irreducible $A$-submodule of $V$,
so $W=Rp^{\ell-2}(x)v_2$. But then the minimal polynomial of $x$
acting on $V/W$ is $p^{\ell-2}$, a contradiction. This proves that
$\ell=2$.

Thus $V$ is a completely reducible $R$-module, so $V=W\oplus U$,
for some irreducible $R$-submodule $U$ of $V$.

But $V$ is not a completely reducible $A$-module, so $A$ is not a
semisimple algebra. Therefore, $A$ has a non-zero nilpotent
element, say $y$. Hence $p_y(X)=X$, so $y$ acts trivially on $W$
and $V/W$. It follows that minimal polynomial of $u=x+y$ is $p$
or~$p^2$. Suppose, if possible, that the first case occurs. Then
$u$ is semisimple and has 2 different Jordan-Chevalley
decompositions, which is impossible since $F$ is perfect (see
\cite{B}, Chapter VII, \S 5). Thus $u$ has minimal polynomial
$p^2$, so $V$ is a uniserial $F[u]$-module, and a fortiori
$A=F[u]$. This proves (a).

Suppose now the equivalent conditions (a)-(c) hold. Since $A$ is
local, it is clear that its residue field $R(A)$ is its unique
irreducible module. Let $N$ be the number of distinct prime
factors of $[R(A):F]$ and suppose that $|F|>N-1$. Let
$x_1,\dots,x_n$ be any elements of $A$ satisfying
$A=F[x_1,\dots,x_n]$. We wish to show the existence of an
$F$-linear combination $z$ of $x_1,\dots,x_n$ such that $A=F[z]$.
We argue by induction on the composition length $\ell$ of $V=A$ as
an $A$-module.

If $\ell=1$ then $A=F[u]$ is a finite field extension of $F$. In
this case the existence of $z$ follows from Theorem \ref{be}.

Suppose next $\ell>1$ and the result is true for algebras of
satisfying (a)-(c) of length $<\ell$. Arguing as above, since $C$
is a non-zero factor of the local ring $A$, its residue field is
$F$-isomorphic to that of $A$. By inductive hypothesis, there is
$v$ in the $F$-span of $x_1,\dots,x_n$ such that
$C=F[\widetilde{v}]$. In Case 1, when $v$ is cyclic, we can take
$z=v$. In Case 2 we have $V=W\oplus U$, where $W$ is the socle of
the $A$-module $V$ and $W,U$ are irreducible $F[v]$-modules upon
which $v$ acts with irreducible minimal polynomial $p_v$. Suppose,
if possible, that every non-zero $x_1,\dots,x_n$ acts semisimply
on $V$. Since $F$ is perfect, it follows that every element of $A$
acts semisimply on $V$ (see \cite{B}, Chapter VII, \S 5, Corollary
to Proposition 16). This contradicts the fact that $A$ has
non-zero nilpotent radical. Thus, there is at least one $i$ such
that $y=x_i\neq 0$ is not semisimple. Since $F$ is perfect, $y$
has a Jordan-Chevalley decomposition in $\mathrm{End}(V)$ (see
\cite{B}, Chapter VII, \S 5), say $y=s+m$, where $s,m\in
F[y]\subseteq A$, $s$ is semisimple, $0\neq m$ is nilpotent, and
$[s,m]=0$. By Theorem \ref{be} applied to the finite field
extension $B=F[v|_W]$ of $F$, there is $0\neq \alpha\in F$ such
that $z=v+\alpha y$ satisfies $B=F[z|_W]$. Since $m$ acts
trivially on $W$, it follows that that $v+\alpha s$ acts
irreducibly on $W$. Now $v$ and $\alpha s$ are semisimple and
commute. Since $F$ is perfect, it follows, as above, that
$v+\alpha s$ is semisimple. By uniqueness of the Jordan-Chevalley
decomposition we see, as above, that $v+\alpha s+\alpha m=z$ has
minimal polynomial~$p_z^2$, where $\deg(p_z^2)=\deg(p^2)=\dim(V)$.
Thus $A=F[z]$, where $z=v+\alpha x_i$ is an $F$-linear combination
of $x_1,\dots,x_n$.
\end{proof}

\begin{cor}\label{t1} Let $F$ be a perfect field and let $A$ be a commutative and associative $F$-algebra.
Let $V$ be a finite dimensional uniserial $A$-module of length
$\ell$. Then there exists $x\in A$ such that $V$ is a uniserial
$F[x]$-module. In particular, $x$ acts on $V$ via the companion
matrix $C_f$ of a power $f=p^\ell$ of an irreducible polynomial
$p\in F[X]$, and every element of $A$ acts on $V$ via a polynomial
on $C_f$.
\end{cor}

\begin{cor}\label{t22} Let $F$ be a perfect field and let $\g$ be an abelian Lie algebra over~$F$.
Let $V$ be a finite dimensional uniserial $L$-module of length
$\ell$. Let $W$ be the socle of $V$.  Let $N$ be the number of
distinct prime factors of $\dim(W)$ and suppose that $|F|>N-1$.
Then there exists $x\in \g$ such that $V$ is a uniserial
$F[x]$-module. In particular, $x$ acts on $V$ via the companion
matrix $C_f$ of a power $f=p^\ell$ of an irreducible polynomial
$p\in F[X]$, and other every element of $\g$ acts on $V$ via a
polynomial on $C_f$.
\end{cor}

\begin{proof} Apply Theorem \ref{abel} to the subalgebra of $\End(V)$ generated by the image of~$\g$ under
the given representation.
\end{proof}

\begin{cor}\label{dernier} Let $F$ be an algebraically closed
field and let $\mathcal A$ be an abelian associative or Lie
algebra over $F$. Let $V$ be a uniserial $\mathcal A$-module of
finite dimension $m$. Then there exists an $x$ in $\mathcal A$
that is represented by a Jordan block $J_m(\alpha)$, $\alpha\in
F$, relative to a basis of $V$. Moreover, every other element of
$\mathcal A$ is represented by a polynomial in $J_m(\alpha)$ in
that basis.
\end{cor}

\begin{note}\label{menti}{\rm It is well-known (see \cite{Ar}, Theorem 27) that if $F$ is a
perfect field then every finite extension $K$ of $F$ has a
primitive element, i.e., an element $x\in K$ such that $K=F[x]$.

However, certain imperfect fields share this property and it is
conceivable that the implication $(c)\rightarrow (a)$ of Theorem
\ref{abel} remains valid for these fields, which is certainly the
case when the given uniserial module is irreducible. The purpose
of this note is to show that the implication $(c)\rightarrow (a)$
of Theorem \ref{abel} actually fails for every imperfect field,
once we allow uniserial non-irreducible modules. In short, the
condition that $F$ be perfect is essential to Theorem \ref{abel}.

For the remainder of this note $F$ stands for a field of prime
characteristic $p$. Then $F^p$ is a subfield of $F$ and the degree
$[F:F^p]$ is either infinite or a power of $p$. Following
Teichm$\rm{\ddot{u}}$ller \cite{T}, we define the degree of
imperfection of $F$ to be infinite or $m\geq 0$, depending on
whether $[F:F^p]$ is infinite or equal to $p^m$. Thus $F$ is
perfect if and only if it has degree of imperfection 0. In this
case, if $K=F(X_1,\dots,X_m)$, where $X_1,\dots,X_m$ are
algebraically independent over $F$, then $K$ has degree of
imperfection $m$.

It was shown by Steinitz \cite{St} that every finite extension of
an imperfect field of degree of imperfection 1 has a primitive
element. For a generalization of Steinitz result see \cite{BM}. It
is easy to see that every imperfect field whose degree of
imperfection is $>1$ has a finite extension with no primitive
element. In particular, Theorem \ref{abel} does not apply to any
of these fields.

Suppose $F$ is imperfect. We claim that there is a commutative and
associative $F$-algebra $A$ of dimension $2p$ with a faithful
uniserial $A$-module $V$ of dimension $2p$ and no $x\in A$ such
$A=F[x]$. The key underlying factor here is that the uniqueness
part of the Jordan-Chevalley decomposition fails over an imperfect
field.

To prove the claim observe that, by hypothesis, there exists $a\in
F$ such that $a\notin F^p$. Therefore (see \cite{M}, \S 1.4,
Theorem 9) the polynomial $f=X^p-a\in F[X]$ is irreducible. Let
$C$ be the companion matrix to $f$ and consider the matrices in
$D,E\in M_{2p}(F)$ defined in terms of $p\times p$ blocks as
follows:
$$
D=\left(%
\begin{array}{cc}
  C & 0 \\
  0 & C \\
\end{array}%
\right),\,
E=\left(%
\begin{array}{cc}
  0 & I_p \\
  0 & 0 \\
\end{array}%
\right).
$$
It is clear that the subalgebra $A$ of $M_{2p}(F)$ generated by
$D$ and $E$ is commutative of dimension $2p$ with $F$-basis
$D^iE^j$, where $0\leq i<p$, $0\leq j<1$. Moreover, the column
space $V=F^{2p}$ is a faithful uniserial $A$-module of dimension
$2p$. Let $x\in A$. Then there are $\alpha_{ij}\in F$ such that
$x$ is the sum of all $\alpha_{ij}D^iE^j$, $0\leq i<p$, $0\leq
j<1$. Since $E^2=0$ and $D^p=a I_{2p}$, it follows that $x^p$ is
the sum of all $\alpha_{i0}^p a^i I_{2p}$, $0\leq i<p$. Thus,
$x^p=s I_{2p}$, where $s\in F$, so $x$ is a root of $X^p-s\in
F[X]$. In particular, $F[x]$ has dimension $\leq p$, so $F[x]$ is
a proper subalgebra of $A$. Note, finally, that $B=D+E$ satisfies
$$B^p=(D+E)^p=D^p+E^p=a I_{2p},
$$
so the minimal polynomial of $B$ is $f$. In particular, $B$ is
semisimple and has 2 different Jordan-Chevalley decompositions.
}
\end{note}

\begin{note} {\rm Let $A$ be a commutative associative algebra
over a field $F$ and let $V$ be a finite dimensional faithful
uniserial $A$-module. It is not difficult to see that $V=Av$ must
be cyclic, where the annihilator of $v$ is trivial. Then $A\cong
Av=V$, so $A$ is uniserial as $A$-module. The argument is taken
from the proof of \cite{SL}, Proposition~2.1. However, the fact
that $A$ itself is uniserial in no way implies the results of this
section, as shown in Note \ref{menti}.}
\end{note}

\begin{note} {\rm The condition $|F|>N-1$ is essential in Theorem
\ref{abel} and Corollary~\ref{t22}, as the irreducible module
$F[x,y]$ of Theorem \ref{unomas} shows.}
\end{note}

\begin{note}{\rm The condition that $V$ be finite dimensional in Theorem \ref{abel} is essential. Let $F$ be a field and let $K=F(X)$.
The regular module of $K$ is irreducible, but there is no $x\in K$
such that $K=F[x]$.}
\end{note}

\begin{note}{\rm The use of uniserial modules in Theorem
\ref{abel} and Corollary~\ref{t22} cannot be extended to more
general indecomposable modules. Indeed, let $F$ be a field and let
$V$ be a finite dimensional indecomposable but not uniserial
module for a Lie or associative algebra. Then there is no $x$ in
this algebra such that $V$ is an indecomposable $F[x]$-module, for
in that case $V$ will be a uniserial $F[x]$-module, and hence a
uniserial module for the given algebra.}
\end{note}

\begin{note}{\rm The theorem of the primitive element
for finite field extensions of perfect fields is a special case of
Theorem \ref{abel}.}
\end{note}

\noindent {\bf Acknowledgment.} We thank A. Herman for helpful
discussions about Theorem \ref{unomas}.


\end{document}